\documentclass[journal,twoside,web]{ieeecolor}
\bstctlcite{bstctl:etal, bstctl:nodash, bstctl:simpurl}

\usepackage{generic}
\usepackage{cite}
\usepackage{amsmath,amssymb,amsfonts,amsthm,balance,bm,mathtools}
\usepackage{algorithmic}
\usepackage{graphicx}
\usepackage{algorithm,algorithmic}
\usepackage{hyperref}
\hypersetup{hidelinks=true}
\usepackage{textcomp}
\usepackage{subfigure}

\newtheorem{theorem}{Theorem}
\newtheorem{lemma}{Lemma}

\newtheorem{remark}{Remark}

\newtheorem{corollary}{Corollary}

\newcommand{\differential}{{\rm{d}}}

\newcommand{\hess}{{\rm{Hess}}}

\newcommand{\RNum}[1]{\uppercase\expandafter{\romannumeral #1\relax}}

\def\BibTeX{{\rm B\kern-.05em{\sc i\kern-.025em b}\kern-.08em
    T\kern-.1667em\lower.7ex\hbox{E}\kern-.125emX}}
\allowdisplaybreaks

\begin{document}
\bstctlcite{IEEE_b:BSTcontrol}

\title{On the Hopf-Cole Transform for Control-affine Schr\"{o}dinger Bridge}

\author{Alexis M.H. Teter and Abhishek Halder, \IEEEmembership{Senior Member, IEEE}% <-this % stops a space
\thanks{Alexis M.H. Teter is with the Department of Applied Mathematics, University of California Santa Cruz, Santa Cruz, CA 95064, USA {\tt\small{amteter@ucsc.edu}}}
\thanks{Abhishek Halder is with the Department of Aerospace Engineering, Iowa State University, Ames, IA 50011, USA, {\tt\small{ahalder@iastate.edu}}.%
}}

\maketitle

%%%%%%%%%%%%%%%%%%%%%%%%%%%%%%%%%%%%%%%%%%%%%%%%%%%%%%%%%%%%%%%%%%%%%%%%%%%%%%%%
\begin{abstract}
The purpose of this note is to clarify the importance of the relation $\bm{gg}^{\top}\propto\bm{\sigma\sigma}^{\top}$ in solving control-affine Schr\"{o}dinger bridge problems via the Hopf-Cole transform, where $\bm{g},\bm{\sigma}$ are the control and noise coefficients, respectively. We show that the Hopf-Cole transform applied to the conditions of optimality for generic control-affine Schr\"{o}dinger bridge problems, i.e., without the assumption $\bm{gg}^{\top}\propto\bm{\sigma\sigma}^{\top}$, gives a pair of forward-backward PDEs that are neither linear nor equation-level decoupled. We explain how the resulting PDEs can be interpreted as nonlinear forward-backward advection-diffusion-reaction equations, where the nonlinearity stem from additional drift and reaction terms involving the gradient of the log-likelihood a.k.a. the score. These additional drift and reaction vanish when  $\bm{gg}^{\top}\propto\bm{\sigma\sigma}^{\top}$, and the resulting boundary-coupled system of linear PDEs can then be solved by dynamic Sinkhorn recursions. A key takeaway of our work is that the numerical solution of the generic control-affine Schr\"{o}dinger bridge requires further algorithmic development, possibly generalizing the dynamic Sinkhorn recursion or otherwise.
\end{abstract}

%%%%%%%%%%%%%%%%%%%%%%%%%%%%%%%%%%%%%%%%%%%%%%%%%%%%%%%%%%%%%%%%%%%%%%%%%%%%%%%%

\section{Introduction}\label{sec:introduction}
\subsubsection*{Schr\"{o}dinger bridge} \emph{Schr\"{o}dinger bridge} is a class of stochastic optimal control problems where the control objective is to steer the state statistics from a given probability density function (PDF) to another over a fixed deadline while minimizing the total control effort. Genesis of this idea occurred about a century ago in the works of Erwin Schr\"{o}dinger \cite{Sch31,Sch32} for controlled Brownian motion. For an English translation of Schr\"{o}dinger's 1931 paper, see \cite{chetrite2021schrodinger}.

Schr\"{o}dinger's motivation in \cite{Sch31,Sch32} was to seek a diffusion process interpretation of quantum mechanics, and these works were not formulated as control problems. This is unsurprising because neither the axiomatic theory of stochastic processes\footnote{Kolmogorov's celebrated Grundbegriffe \cite{kolmogoroff1933grundbegriffe} appeared in 1933.} nor the discipline of feedback control\footnote{``Control emerged as a discipline after the Second World War" \cite{aastrom2014control}.} existed in 1931-32. The understanding that the Schr\"{o}dinger bridges are stochastic optimal controlled processes emerged in the late 20th century \cite{follmer1988random,dai1991stochastic,pavon1991free,blaquiere1992controllability,chen2016relation,chen2021stochastic}. Since then, natural control-theoretic generalizations of the classical Schr\"{o}dinger bridge (i.e., controlled Brownian motion) have appeared. These include additive state-cost-to-go \cite{dawson1990schrodinger,wakolbinger1990schrodinger,leonard2011stochastic,teter2024schr,teter2024weyl,teter2025probabilistic}, nonlinear controlled stochastic differential equations \cite{chen2015fast,haddad2020prediction,caluya2020finite,caluya2021wasserstein,nodozi2022schrodinger,behjoo2025harmonic}, additional path constraints \cite{caluya2021reflected}, etc.

\subsubsection*{Hopf-Cole transform} A popular approach to numerically solve the (classical and generalized) Schr\"{o}dinger bridge problems is to use the \emph{Hopf-Cole transform} \cite{hopf1950partial,cole1951quasi} originally developed for problems in fluid dynamics. In the Schr\"{o}dinger bridge context, this transform is a bijective change-of-variable of the form:\\
\noindent (optimal joint PDF, value function) $\mapsto$ Schr\"{o}dinger factors,
where the Schr\"{o}dinger factors are a suitable pair of measurable functions of state and time. 

The pair (optimal joint PDF, value function) is the primal-dual pair for the Schr\"{o}dinger bridge problem. They satisfy a primal (controlled forward Kolmogorov) PDE and a dual (Hamilton-Jacobi-Bellman-like) PDE, respectively, that are coupled, nonlinear, and are subject to two boundary conditions. It turns out that this coupled PDE system is problematic for computation since both boundary conditions are available in the primal variable, and none in the dual variable. 

When applicable, the main appeal of the Hopf-Cole transform is that in the new variable pair (Schr\"{o}dinger factors), the PDEs become \emph{equation-level decoupled and linear}. After the transform, the coupling gets pushed to the boundary conditions. The resulting system of boundary-coupled linear PDEs can then be solved by dynamic Sinkhorn recursions \cite{marino2020optimal} that are contractive with guaranteed worst-case linear rate of convergence \cite{chen2016entropic,teter2023contraction}.

Recently, it has been pointed out \cite{nodozi2023neural} that the Hopf-Cole transform is not applicable for \emph{control non-affine} Schr\"{o}dinger bridge problems where both the drift and diffusion coefficients are non-affine in control. Then, the necessary conditions of optimality yield a system of $m+2$ coupled nonlinear PDEs (as opposed to 2 primal-dual PDEs) where $m$ is the number of control inputs. The solution of the coupled PDE system in \cite{nodozi2023neural} was approximated by training a custom neural network. Motivation for such control non-affine Schr\"{o}dinger bridge came from material synthesis using controlled colloidal self-assembly \cite{nodozi2023physics}.

\subsubsection*{Contributions} We explain why the Hopf-Cole transform does not reap computational benefit even for \emph{control-affine} Schr\"{o}dinger bridge problems unless the input and the noise channels coincide. When these channels are non-identical, we prove that the Schr\"{o}dinger factor PDEs obtained from the Hopf-Cole transform, remain coupled and nonlinear due to \emph{extra drift and reaction terms that are nonlinear in a Schr\"{o}dinger factor}. These extra terms are shown to vanish precisely when the input and noise channels are the same, thereby fundamentally changing the structure of the PDE system (from coupled nonlinear to decoupled linear). 

To the best of our knowledge, the PDEs and structural results we derive, are new. They have likely escaped attention so far because existing works \cite{chen2021stochastic,caluya2021wasserstein} made the identical channel assumption \emph{a priori}. These assumptions, though mathematically rather restrictive, were nonetheless motivated by practical reasons (identical input and noise channels indeed occur for stochastic forcing and noisy actuators) or ease of calculations. Removing these assumptions, we uncover a generalized PDE system.

The main takeaway of our results is that new computational algorithms--beyond standard Sinkhorn recursions--are needed to solve generic control-affine Schr\"{o}dinger bridge problems. This will be pursued in our future work.

\subsubsection*{Organization} The problem statement for the control-affine Schr\"{o}dinger bridge appears in Sec. \ref{subsec:phiphihatPDEs} (equation \eqref{CASBP}). Sec. \ref{sec:ConditionsForOptimality} provides the necessary conditions of optimality (Theorem \ref{Thm:1stOrderConditionsOfOptimality}) as a system of primal-dual PDEs. The Schr\"{o}dinger factor PDEs for generic control-affine case are derived in Sec. \ref{subsec:phiphihatPDEs} (Theorem \ref{Thm:ForwardBackward}) with associated interpretations in Sec. \ref{subsec:Interpretations} and implications in Sec. \ref{subsec:Implications}. Sec. \ref{sec:conclusions} concludes the paper.

%%%%%%%%%%%%%%%%%%%%%%%%%%%%%%%%%%%%%%%%%%%%%%%%%%%

\section{Control-affine Schr\"{o}dinger Bridge Problem}\label{sec:CASBP}
Consider an It\^{o} diffusion process governing the controlled state vector $\bm{x}_{t}^{\bm{u}}\in\mathbb{R}^{n}$ over a prescribed finite time horizon $[t_0,t_1]$, of the form 
\begin{align}
\differential\bm{x}_{t}^{\bm{u}} = \left(\bm{f}\left(t,\bm{x}_{t}^{\bm{u}}\right) + \bm{g}\left(t,\bm{x}_{t}^{\bm{u}}\right)\bm{u}\right)\differential t + \bm{\sigma}(t,\bm{x}_{t}^{\bm{u}})\differential\bm{w}_t,
\label{ControlledSDE}    
\end{align}
where $\bm{w}_t\in\mathbb{R}^{p}$ is standard Brownian motion, and the control $\bm{u}\in\mathcal{U}:=\{\bm{v}:[t_0,t_1]\times\mathbb{R}^{n}\mapsto\mathbb{R}^{m} \mid \|\bm{v}\|_2^2 < \infty\}$, the set of Markovian finite energy inputs. Notice that the controlled drift $\bm{f} + \bm{gu}$ is affine in control.

Let the diffusion tensor 
\begin{align}
\bm{\Sigma}:=\bm{\sigma}\bm{\sigma}^{\top}\succeq\bm{0}.
\label{defDiffusionTensor}    
\end{align}
We make the following standard regularity assumptions on the unforced drift coefficient $\bm{f}$, and the diffusion coefficient $\bm{\sigma}$. Here and throughout, the notation $\langle\cdot,\cdot\rangle$ stands for the standard Euclidean inner product for vector arguments, and the Frobenius a.k.a. the Hilbert-Schmidt inner product for matrix arguments. 
\begin{itemize}
\item[\textbf{A1.}] (\textbf{Non-explosion and Lipschitz coefficients}) There exist constants $c_1,c_2>0$ such that $\forall\bm{x},\bm{y}\in\mathbb{R}^n$, $\forall t\in[t_0,t_1]$, 
\begin{align*}
\|\bm{f}(t,\bm{x})\|_2 + \|\bm{\sigma}\left(t,\bm{x}\right)\|_{2} 
&\leq 
c_1\left(1 + \|\bm{x}\|_2\right), \\
\|\bm{f}(t,\bm{x}) - \bm{f}(t,\bm{y})\|_2 
&\leq 
c_2 \|\bm{x}-\bm{y}\|_2.    
\end{align*}
\item[\textbf{A2.}] (\textbf{Uniformly lower bounded diffusion tensor}) There exists a constant $c_3>0$ such that $\forall \bm{x}\in\mathbb{R}^{n}$, $\forall t\in[t_0,t_1]$,
\begin{align*}
    \langle \bm{x}, \bm{\Sigma}(t,\bm{x})\bm{x}\rangle 
    \geq 
    c_3 \|\bm{x}\|_2^2.
\end{align*}  
\end{itemize}
The assumptions \textbf{A1}-\textbf{A2} ensure the strong solution \cite[p. 66]{oksendal2013stochastic} for the uncontrolled ($\bm{u}\equiv\bm{0}$) version of \eqref{ControlledSDE}, and also that the corresponding transition density is everywhere continuous and positive \cite[p. 364]{karatzas2014brownian}.

Our interest is to solve a control-affine Schr\"{o}dinger Bridge problem: a stochastic optimal control problem to minimize an expected cost-to-go while transferring the controlled state between given endpoint joint state PDFs $\rho_0,\rho_1$ subject to the controlled dynamics \eqref{ControlledSDE} and hard deadline constraints. For $t\in[t_0,t_1]$, denote the space of continuous (w.r.t. $t$) PDF-valued curves with fixed endpoints $\rho_0,\rho_1$, as
\begin{align}
\mathcal{P}_{01} := &\{t\mapsto \rho(t,\cdot)\,\text{continuous} \mid \rho \geq 0, \!\!{\small{\int}}\!\!\rho(t,\cdot) = 1\forall t\in\![t_0,t_1],\nonumber \\ &\rho(t_0,\cdot) = \rho_0(\cdot), \rho(t_1,\cdot) = \rho_1(\cdot)\}.
\label{defP01}    
\end{align}
Then, the control-affine Schr\"{o}dinger Bridge problem is
\begin{subequations}
\begin{align}
&\underset{\left(\rho^{\bm{u}},\bm{u}\right)\in\mathcal{P}_{01}\times\mathcal{U}}{\arg\inf}\int_{t_0}^{t_1}\mathbb{E}_{\rho^{\bm{u}}}\left[q\left(t,\bm{x}_{t}^{\bm{u}}\right) + \frac{1}{2}\|\bm{u}\|_2^2\right]\differential t
\label{CASBPobjective}\\
&\text{subject to} \nonumber \\ &\quad\dfrac{\partial\rho^{\bm{u}}}{\partial t} + \nabla_{\bm{x}_t^{\bm{u}}}\cdot\left(\rho^{\bm{u}}\left(\bm{f}\left(t,\bm{x}_{t}^{\bm{u}}\right)+\bm{g}\left(t,\bm{x}_{t}^{\bm{u}}\right)\bm{u}\right)\right) \nonumber \\ & \qquad \qquad\qquad \qquad \qquad \qquad = \frac{1}{2}\Delta_{\bm{\Sigma}\left(t,\bm{x}_{t}^{\bm{u}}\right)}\:\rho^{\bm{u}},\label{CASBPpdeconstraint}
\end{align}
\label{CASBP}
\end{subequations}
where the weighted Laplacian
\begin{align}
\Delta_{\bm{\Sigma}(t,\bm{x})}\:\rho := \sum_{i,j=1}^{n}\dfrac{\partial^{2}}{\partial x_{i} \partial x_{j}}\left(\bm{\Sigma}_{ij}(t,\bm{x})\rho(t,\bm{x})\right).
\label{DefWeightedLaplacian}    
\end{align}
The objective \eqref{CASBPobjective} is an expected additive cost comprising of a state cost-to-go and control effort. The constraint \eqref{CASBPpdeconstraint} is the controlled Fokker-Planck-Kolmogorov PDE governing the dynamics of the controlled PDF. In other words, \eqref{CASBPpdeconstraint} is the macroscopic dynamics associated with the controlled sample path dynamics \eqref{ControlledSDE}.

Assuming \textbf{A1}-\textbf{A2} holds, the endpoint PDFs $\rho_0,\rho_1$ have finite second moments, and the state cost $q$ suitably smooth, the existence-uniqueness for \eqref{CASBP} can be established \cite{wakolbinger1990schrodinger}.

%%%%%%%%%%%%%%%%%%%%%%%%%%%%%%%%%%%%%%%%%%%%%%%%%%%%%%%%%%%%%%%%%

\section{Necessary Conditions for Optimality}\label{sec:ConditionsForOptimality}
The Theorem \ref{Thm:1stOrderConditionsOfOptimality} next derives the necessary conditions for optimality for problem \eqref{CASBP}.

\begin{theorem}[Coupled Nonlinear PDEs with linear boundary conditions]\label{Thm:1stOrderConditionsOfOptimality}
Let \textbf{A1}-\textbf{A2} hold. Consider two PDFs $\rho_0,\rho_1$ having finite second moments, a given time horizon $[t_0,t_1]$, and the state cost $q$ suitably smooth. For $S\in\mathcal{C}^{1,2}\left([t_0,t_1];\mathbb{R}^{n}\right)$, the necessary conditions for optimality for problem \eqref{CASBP} are
\begin{subequations}
\begin{align}
%&\text{(primal PDE)}
\dfrac{\partial\rho^{\bm{u}}_{\mathrm{opt}}}{\partial t} + \nabla_{\bm{x}}\cdot\left(\rho^{\bm{u}}_{\mathrm{opt}}\left(\bm{f}+\bm{g}\bm{g}^{\top}\nabla_{\bm{x}}S\right)\right) = \frac{1}{2}\Delta_{\bm{\Sigma}}\:\rho^{\bm{u}}_{\mathrm{opt}}, \nonumber
 \\ \label{PrimalPDE} \textrm{(primal PDE)}  \\ \dfrac{\partial S}{\partial t} +\langle\nabla_{\bm{x}}S,\bm{f}\rangle +\frac{1}{2}\langle\nabla_{\bm{x}}S,\bm{gg}^{\top}\nabla_{\bm{x}}S\rangle +\frac{1}{2}\langle\bm{\Sigma},\hess_{\bm{x}}S\rangle = q,\nonumber \\ 
 \text{(dual PDE)} \label{DualPDE}\\
%&\text{(primal boundary conditions)} \\ & 
\rho^{\bm{u}}_{\mathrm{opt}}\left(t=t_0,\cdot\right)=\rho_0(\cdot), \quad \rho^{\bm{u}}_{\mathrm{opt}}\left(t=t_1,\cdot\right)=\rho_1(\cdot). \nonumber  \\
\label{PrimalBC} (\text{primal boundary conditions})
\end{align}
\label{FOCO}    
\end{subequations}
The optimal control is
\begin{align}
\bm{u}_{\mathrm{opt}} = \bm{g}^{\top}\nabla_{\bm{x}}S.
\label{OptimalControl}    
\end{align}
\end{theorem}
\begin{proof}
Consider the Lagrangian $\mathcal{L}$ for the constrained variational problem \eqref{CASBP}:
\begin{align}
\mathcal{L}\left(\rho^{\boldsymbol{u}},\bm{u},S\right) :=& \int_{t_0}^{t_1} \int_{\mathbb{R}^{n}} \left(q+\frac{1}{2}\|\boldsymbol{u}\|_2^2\right) \rho^{\boldsymbol{u}} \mathrm{d}\boldsymbol{x} \mathrm{~d} t \nonumber \\  &+\int_{t_0}^{t_1} \int_{\mathbb{R}^{n}} S \frac{\partial \rho^{\boldsymbol{u}}}{\partial t} \mathrm{~d} \boldsymbol{x} ~\differential t \nonumber \\ &+ \int_{t_0}^{t_1} \int_{\mathbb{R}^{n}}\left(\nabla_{\bm{x}}\cdot\left(\bm{f}+\bm{gu}\right)\right)S\differential\bm{x}~\differential t\nonumber\\
&- \frac{1}{2}\int_{t_0}^{t_1} \int_{\mathbb{R}^{n}} \left(\Delta_{\bm{\Sigma}}\rho^{\bm{u}}\right) S\differential\bm{x}~\differential t,
\label{defLagrangian}
\end{align}
where the Lagrange multiplier $S\in\mathcal{C}^{1,2}\left([t_0,t_1];\mathbb{R}^{n}\right)$. The idea now is to perform integration-by-parts for the second, third and fourth summands in \eqref{defLagrangian} such that the differential operators acting on $\rho^{\bm{u}}$ gets transferred to those of $S$. This will then allow us to re-write \eqref{defLagrangian} as a linear functional in $\rho^{\bm{u}}$ only.

Specifically, for the second summand in \eqref{defLagrangian}, we apply the Fubini-Torreli theorem to switch the order of integration and perform integration-by-parts w.r.t. $t$, to obtain
\begin{align}
&\int_{t_0}^{t_1} \int_{\mathbb{R}^{n}} S \frac{\partial \rho^{\boldsymbol{u}}}{\partial t} \mathrm{~d} \boldsymbol{x} ~\differential t \nonumber  \\ &= \underbrace{\int_{\mathbb{R}^{n}}\left(S(t_1,\bm{x})\rho_{1}(\bm{x}) - S(t_0,\bm{x})\rho_0(\bm{x})\right)\differential\bm{x}}_{\text{constant over}\,\mathcal{P}_{01}\times\mathcal{U}} \nonumber \\  &\qquad -\int_{t_0}^{t_1} \int_{\mathbb{R}^{n}}\dfrac{\partial S}{\partial t}\rho^{\bm{u}}\differential\bm{x}~\differential t.
\label{IBP2ndSummand}    
\end{align}

For the third summand in \eqref{defLagrangian}, integration-by-parts w.r.t. $\bm{x}$, and assuming the limits at $\|\bm{x}\|_2\rightarrow\infty$ are zero, we find
\begin{align}
\int_{t_0}^{t_1} \int_{\mathbb{R}^{n}}&\left(\nabla_{\bm{x}}\cdot\left(\bm{f}+\bm{gu}\right)\right)S\differential\bm{x}~\differential t \nonumber \\ &= -\int_{t_0}^{t_1}\int_{\mathbb{R}^{n}}\langle\nabla_{\bm{x}}S,\bm{f}+\bm{gu}\rangle\rho^{\bm{u}}\differential\bm{x}~\differential t.
\label{IBP3rdSummand}    
\end{align}

For the fourth summand in \eqref{defLagrangian},
two-fold integration by parts w.r.t. $\bm{x}$ yields
\begin{align}
\int_{\mathbb{R}^{n}} \left(\Delta_{\bm{\Sigma}}\rho^{\bm{u}}\right) S\differential\bm{x}~\differential t = \int_{\mathbb{R}^{n}} \langle\bm{\Sigma},\hess_{\bm{x}}S\rangle\rho^{\bm{u}}\differential\bm{x}.
\label{IBP4thSummand}    
\end{align}

Substituting \eqref{IBP2ndSummand}, \eqref{IBP3rdSummand}, \eqref{IBP4thSummand} in \eqref{defLagrangian}, we obtain
\begin{align}
&\mathcal{L}\left(\rho^{\boldsymbol{u}},\bm{u},S\right) = \text{constant}\;+\;\int_{t_0}^{t_1} \int_{\mathbb{R}^{n}} \left(q+\frac{1}{2}\|\boldsymbol{u}\|_2^2 -\dfrac{\partial S}{\partial t} \right. \nonumber \\ &\qquad\left. -\langle\nabla_{\bm{x}}S,\bm{f}+\bm{gu}\rangle -\frac{1}{2}\langle\bm{\Sigma},\hess_{\bm{x}}S\rangle\right)\rho^{\bm{u}}\differential\bm{x}~\differential t.
\label{LagrangianSimplified}    
\end{align}
Minimization of \eqref{LagrangianSimplified} w.r.t. $\bm{u}$ gives the optimal control \eqref{OptimalControl}.

Substituting \eqref{OptimalControl} in the primal feasibility constraint \eqref{CASBPpdeconstraint}, we arrive at \eqref{PrimalPDE}.

Substituting \eqref{OptimalControl} back into \eqref{LagrangianSimplified}, and requiring the integral to be zero, yields the dynamic programming equation:
\begin{align}
\int_{t_0}^{t_1} \int_{\mathbb{R}^{n}} \left(q-\dfrac{\partial S}{\partial t} -\langle\nabla_{\bm{x}}S,\bm{f}\rangle -\frac{1}{2}\langle\nabla_{\bm{x}}S,\bm{gg}^{\top}\nabla_{\bm{x}}S\rangle \right. \nonumber \\ \left. -\frac{1}{2}\langle\bm{\Sigma},\hess_{\bm{x}}S\rangle\right)\rho^{\bm{u}}\differential\bm{x}~\differential t = 0.
\label{DPequation}    
\end{align}
Since \eqref{DPequation} should be satisfied for arbitrary $\rho^{\bm{u}}$, we must have
\eqref{DualPDE}.

The restriction of $\rho^{\bm{u}}$ to $\mathcal{P}_{01}$ in \eqref{CASBPobjective} yields \eqref{PrimalBC}.
\end{proof}

\begin{remark}
The derived necessary conditions \eqref{FOCO} comprise of two coupled nonlinear PDEs \eqref{PrimalPDE}-\eqref{DualPDE} in the primal-dual variable pair $\left(\rho^{\bm{u}}_{\mathrm{opt}},S\right)$ with two linear boundary conditions \eqref{PrimalBC}. What makes this system atypical is that both of these boundary conditions are available for the single variable $\rho^{\bm{u}}_{\mathrm{opt}}$. This juxtaposes with classical stochastic optimal control where a terminal cost  would generate a boundary condition on $S$ which would in turn make the resulting system relatively straightforward for computation. In contrast, enforcing state statistics at the two endpoints yield \eqref{PrimalBC}, thereby causing unavailability of boundary condition on $S$.   
\end{remark}

\begin{remark}
We note that Theorem \ref{Thm:1stOrderConditionsOfOptimality} does not require the input matrix $\bm{g}$ and the diffusion coefficient $\bm{\sigma}$ to be related in any way. Indeed, the $m$-dimensional input and the $p$-dimensional noise channels in \eqref{ControlledSDE} are unrelated so far.    
\end{remark}

%%%%%%%%%%%%%%%%%%%%%%%%%%%%%%%%%%%%%%%%%%%%%%%%%%%%%%%%%%%%%%%%%

\section{Hopf-Cole Transform and Boundary-coupled Forward-Backward PDEs}\label{sec:HopfColeKolomogorov}
For the primal-dual pair $\left(\rho^{\bm{u}}_{\mathrm{opt}},S\right)$, we next perform a change-of-variables using the Hopf-Cole transform \cite{hopf1950partial,cole1951quasi}:
\begin{align}
\left(\rho^{\bm{u}}_{\mathrm{opt}},S\right)\mapsto\left(\widehat{\varphi},\varphi\right) := \left(\rho^{\bm{u}}_{\mathrm{opt}}\exp\left(-S/\lambda\right),\exp\left(S/\lambda\right)\right), \;\lambda>0.
\label{HopfCole}
\end{align}
The constant $\lambda>0$ is a scaling parameter. 

The appeal of \eqref{HopfCole} is that under certain conditions, it converts \eqref{FOCO}, a system of coupled PDEs with linear boundary conditions, to a system of decoupled PDEs with bilinear boundary conditions. This is summarized in Corollary \ref{Corollary:DecoupledLinearPDEs} which follows from the more general conclusion in Theorem \ref{Thm:ForwardBackward}. The latter is the main result of this section.

We mention here that the $S\mapsto\varphi$ part of \eqref{HopfCole}, i.e., 
$$\varphi = \exp(S/\lambda),$$
is sometimes referred to as Fleming's logarithmic transform \cite{fleming2005logarithmic}, and has been used in general stochastic control \cite{fleming1983stochastic,dai1991stochastic,kappen2005path,todorov2009efficient} and filtering \cite{fleming1982optimal}. For connections between such transforms and the large deviations principle, see e.g., \cite{chetrite2015variational,bierkens2016linear}. Variants of Schr\"{o}dinger bridge problems also admit large deviation interpretations; see e.g., \cite{dawson1990schrodinger,leonard2011stochastic,chen2021stochastic,teter2024schr,teter2025probabilistic}.

\subsection{Derivation for the PDEs in $\left(\widehat{\varphi},\varphi\right)$}\label{subsec:phiphihatPDEs}

Lemma \ref{LemmaWeightedLaplacian} stated next is needed to deduce Theorem \ref{Thm:ForwardBackward} that follows.

\begin{lemma}[Weighted Laplacian of a product]\label{LemmaWeightedLaplacian}
Given scalar fields $\alpha,\beta:\mathbb{R}^{n}\mapsto\mathbb{R}$ and symmetric matrix field $\bm{\Sigma}:\mathbb{R}^{n}\mapsto\mathbb{S}^{n}$ such that $\alpha,\beta,\bm{\Sigma}\in\mathcal{C}^{2}\left(\mathbb{R}^{n}\right)$, we have
\begin{align}
\Delta_{\bm{\Sigma}}\left(\alpha\beta\right) &=  \alpha\langle\bm{\Sigma},\hess_{\bm{x}}\beta\rangle + \beta\langle\bm{\Sigma},\hess_{\bm{x}}\alpha\rangle \nonumber \\ &+ 2\alpha\langle\nabla_{\bm{x}}\cdot\bm{\Sigma},\nabla_{\bm{x}}\beta\rangle + 2\beta\langle\nabla_{\bm{x}}\cdot\bm{\Sigma},\nabla_{\bm{x}}\alpha\rangle\nonumber\\
&+ \alpha\beta\underbrace{\langle \hess_{\bm{x}},\bm{\Sigma}\rangle}_{=\bm{1}^{\top}\bm{\Sigma}\bm{1}} + 2\underbrace{\langle\nabla_{\bm{x}}\alpha,\bm{\Sigma}\nabla_{\bm{x}}\beta\rangle}_{=\langle\nabla_{\bm{x}}\beta,\bm{\Sigma}\nabla_{\bm{x}}\alpha\rangle},
\label{WeightedLaplacianOfProduct}
\end{align}
wherein the divergence of a matrix field is understood as a vector with elements
$$\left(\nabla_{\bm{x}}\cdot\bm{\Sigma}\right)_{i} := \displaystyle\sum_{j=1}^{n}\dfrac{\partial\Sigma_{ij}}{\partial x_{j}} \quad\forall i\in\{1,\hdots,n\}.$$
The following special cases of \eqref{WeightedLaplacianOfProduct} are also of interest:
\begin{align}
&\text{(the case $\beta\equiv 1$)}\nonumber\\&\quad\Delta_{\bm{\Sigma}}\:\alpha = \alpha\langle \hess_{\bm{x}},\bm{\Sigma}\rangle + \langle\bm{\Sigma},\hess_{\bm{x}}\alpha\rangle + 2\langle\nabla_{\bm{x}}\cdot\bm{\Sigma},\nabla_{\bm{x}}\alpha\rangle ,\label{WeightedLaplacianSingleFunction}\\
&\text{(the case $\bm{\Sigma}\equiv\bm{I}$)}\nonumber\\ &\quad\Delta_{\bm{x}}\left(\alpha\beta\right) = \alpha\Delta_{\bm{x}} \beta + \beta\Delta_{\bm{x}} \alpha + 2\langle\nabla_{\bm{x}}\alpha,\nabla_{\bm{x}}\beta\rangle.\label{UnweightedLaplacianOfProduct}
\end{align}
\end{lemma}
\begin{proof}
From \eqref{DefWeightedLaplacian}, we find
\begin{align}
\Delta_{\bm{\Sigma}}\left(\alpha\beta\right) &= \sum_{i,j=1}^{n}\dfrac{\partial}{ \partial x_j}\left(\dfrac{\partial}{\partial x_i}\left(\alpha\beta\Sigma_{ij}\right)\right)\nonumber\\
&= \sum_{i,j=1}^{n}\bigg\{\alpha\beta\dfrac{\partial^2\Sigma_{ij}}{\partial x_i \partial x_j} + \Sigma_{ij}\alpha\dfrac{\partial^2\beta}{\partial x_i\partial x_j} + \Sigma_{ij}\beta\dfrac{\partial^2\alpha}{\partial x_i\partial x_j} \nonumber\\
&\qquad \qquad+ \alpha\left(\dfrac{\partial\beta}{\partial x_{j}}\dfrac{\partial\Sigma_{ij}}{\partial x_i} + \dfrac{\partial\beta}{\partial x_{i}}\dfrac{\partial\Sigma_{ij}}{\partial x_j}\right) \nonumber\\
&\qquad\qquad+ \beta\left(\dfrac{\partial\alpha}{\partial x_{j}}\dfrac{\partial\Sigma_{ij}}{\partial x_i} + \dfrac{\partial\alpha}{\partial x_{i}}\dfrac{\partial\Sigma_{ij}}{\partial x_j}\right) \nonumber\\
&\qquad \qquad+ \Sigma_{ij}\left(\dfrac{\partial\alpha}{\partial x_i}\dfrac{\partial\beta}{\partial x_j} + \dfrac{\partial\alpha}{\partial x_j}\dfrac{\partial\beta}{\partial x_i}\right)\bigg\},
\label{WeightedLaplacianOfProductElementwise}    
\end{align}
which yields \eqref{WeightedLaplacianOfProduct}. The identities \eqref{WeightedLaplacianSingleFunction}-\eqref{UnweightedLaplacianOfProduct} follow by substituting $g\equiv 1$ and $\bm{\Sigma}\equiv\bm{I}$, respectively, in \eqref{WeightedLaplacianOfProduct}. 
\end{proof}

\begin{theorem}[Boundary-coupled nonlinear PDEs]\label{Thm:ForwardBackward}
Consider the setting as in Theorem \ref{Thm:1stOrderConditionsOfOptimality}. Let the pair $\left(\rho^{\bm{u}}_{\mathrm{opt}},S\right)$ solve \eqref{FOCO}. The pair $\left(\widehat{\varphi},\varphi\right)$ given by \eqref{HopfCole} solves
\begin{subequations}
\begin{align}
&\dfrac{\partial\widehat{\varphi}}{\partial t} + \nabla_{\bm{x}}\cdot\left(\widehat{\varphi}\left(\bm{f}+\bm{f}_{\varphi}\right)\right) - \dfrac{1}{2}\Delta_{\bm{\Sigma}}\:\widehat{\varphi} + \left(\dfrac{q}{\lambda}+q_{\varphi}\right)\widehat{\varphi} = 0, \label{varphihatPDE}\\ 
&\dfrac{\partial\varphi}{\partial t} + \langle\nabla_{\bm{x}}\varphi,\bm{f}+\bm{f}_{\varphi}\rangle + \dfrac{1}{2}\langle\bm{\Sigma},\hess_{\bm{x}}\varphi\rangle -  \left(\dfrac{q}{\lambda}+q_{\varphi}\right)\varphi = 0, \label{varphiPDE}\\
&\widehat{\varphi}\left(t_0,\cdot\right)\varphi\left(t_0,\cdot\right)=\rho_0(\cdot), \quad \widehat{\varphi}\left(t_1,\cdot\right)\varphi\left(t_1,\cdot\right)=\rho_1(\cdot), \label{BilinearBC}
\end{align}
\label{BoundaryCoupledForwardBackward}   
\end{subequations}
where
\begin{align}
\bm{f}_{\varphi} &:= \left(\lambda\bm{gg}^{\top}-\bm{\Sigma}\right)\nabla_{\bm{x}}\log\varphi, \label{def_fphi}\\
q_{\varphi} &:= \dfrac{1}{2}\left(\nabla_{\bm{x}}\log\varphi\right)^{\top}\left(\lambda\bm{gg}^{\top}-\bm{\Sigma}\right)\nabla_{\bm{x}}\log\varphi. \label{def_qphi}
\end{align}
The minimizing pair $\left(\rho^{\bm{u}}_{\mathrm{opt}},\bm{u}_{\mathrm{opt}}\right)$ for problem \eqref{CASBP} is then recovered from
\begin{align}
\rho^{\bm{u}}_{\mathrm{opt}}(t,\cdot) &= \widehat{\varphi}(t,\cdot)\varphi(t,\cdot),\label{OptimalPDFfromFactors}\\
\bm{u}_{\mathrm{opt}}(t,\cdot) &= \lambda\:\bm{g}^{\top}\nabla_{(\cdot)}\log\varphi(t,\cdot),\label{OptimalControlfromFactors}
\end{align}
for all $t\in[t_0,t_1]$.
\end{theorem}
\begin{proof}
Since $S\in\mathcal{C}^{1,2}\left([t_0,t_1];\mathbb{R}^{n}\right)$, the $S$ is bounded over its domain, and the mapping \eqref{HopfCole} is bijective. In particular, $\widehat{\varphi}(t,\cdot),\varphi(t,\cdot)$ are positive for all $t\in[t_0,t_1]$.

Multiplying $\widehat{\varphi}(t,\cdot)$ and $\varphi(t,\cdot)$ from \eqref{HopfCole} yields \eqref{OptimalPDFfromFactors}. Substituting $S = \lambda\log\varphi$ in \eqref{OptimalControl} yields \eqref{OptimalControlfromFactors}. Specializing the same for $t=t_0,t_1$, and using \eqref{PrimalBC}, we obtain \eqref{BilinearBC}.

What remains is to deduce \eqref{varphihatPDE}-\eqref{varphiPDE}. Substituting $$S=\lambda\log\varphi$$ in \eqref{DualPDE}, using the Hessian identity 
\begin{align}
\hess_{\bm{x}}\left(\lambda\log\varphi\right) &= \lambda\nabla_{\bm{x}}\circ\nabla_{\bm{x}}\log\varphi \nonumber \\ &= -\dfrac{\lambda}{\varphi^2}\left(\nabla_{\bm{x}}\varphi\right)\left(\nabla_{\bm{x}}\varphi\right)^{\top} + \dfrac{\lambda}{\varphi}\hess_{\bm{x}}\varphi,
\label{HessLogIdentity}    
\end{align}
followed by algebraic simplification, gives 
\begin{align}
\dfrac{\partial\varphi}{\partial t} + \langle\nabla_{\bm{x}}\varphi,\bm{f}\rangle + \dfrac{1}{2}\langle\bm{\Sigma},\hess_{\bm{x}}\varphi\rangle -  \dfrac{q\varphi}{\lambda} \nonumber \\ \qquad + \dfrac{1}{2}\langle\nabla_{\bm{x}}\varphi,\left(\lambda\bm{gg}^{\top}-\bm{\Sigma}\right)\nabla_{\bm{x}}\log\varphi\rangle = 0.
\label{varphiPDEsemifinal}
\end{align}
We next add and subtract (another) $$\dfrac{1}{2}\langle\nabla_{\bm{x}}\varphi,\left(\lambda\bm{gg}^{\top}-\bm{\Sigma}\right)\nabla_{\bm{x}}\log\varphi\rangle$$to the LHS of \eqref{varphiPDEsemifinal}, then group the resulting $+\langle\nabla_{\bm{x}}\varphi,\left(\lambda\bm{gg}^{\top}-\bm{\Sigma}\right)\nabla_{\bm{x}}\log\varphi\rangle$ with the drift term in \eqref{varphiPDEsemifinal}, and the 
\begin{align*}
    &-\dfrac{1}{2}\langle\nabla_{\bm{x}}\varphi,\left(\lambda\bm{gg}^{\top}-\bm{\Sigma}\right)\nabla_{\bm{x}}\log\varphi\rangle \\ &\qquad = -\dfrac{\varphi}{2}\langle\nabla_{\bm{x}}\log\varphi,\left(\lambda\bm{gg}^{\top}-\bm{\Sigma}\right)\nabla_{\bm{x}}\log\varphi\rangle
\end{align*} 
with the reaction term in \eqref{varphiPDEsemifinal}. This gives \eqref{varphiPDE}.

To derive \eqref{varphihatPDE}, we substitute \eqref{OptimalPDFfromFactors} and \eqref{OptimalControlfromFactors} in the primal PDE \eqref{PrimalPDE}. These substitutions allow us to write the LHS of \eqref{PrimalPDE} as:
\begin{align}
{\mathrm{LHS}} &= \widehat{\varphi}\dfrac{\partial\varphi}{\partial t} + \varphi\dfrac{\partial\widehat{\varphi}}{\partial t} + \nabla_{\bm{x}}\cdot\left(\varphi\widehat{\varphi}\bm{f}\right) + \nabla_{\bm{x}}\cdot\left(\widehat{\varphi}\:\lambda\bm{gg}^{\top}\nabla_{\bm{x}}\varphi\right)\nonumber\\
&= \widehat{\varphi}\bigg\{-\langle\nabla_{\bm{x}}\varphi,\bm{f}\rangle - \dfrac{1}{2}\langle\bm{\Sigma},\hess_{\bm{x}}\varphi\rangle +  \dfrac{q\varphi}{\lambda} \nonumber \\ &- \dfrac{1}{2}\langle\nabla_{\bm{x}}\varphi,\left(\lambda\bm{gg}^{\top} -\bm{\Sigma}\right)\nabla_{\bm{x}}\log\varphi\rangle\bigg\} + \varphi\dfrac{\partial\widehat{\varphi}}{\partial t} \nonumber\\
&\quad + \varphi\langle\nabla_{\bm{x}}\widehat{\varphi},\bm{f}\rangle + \widehat{\varphi}\langle\nabla_{\bm{x}}\varphi,\bm{f}\rangle + \varphi\widehat{\varphi}\nabla_{\bm{x}}\cdot\bm{f} \nonumber \\ & + \langle\nabla_{\bm{x}}\widehat{\varphi},\lambda\bm{gg}^{\top}\nabla_{\bm{x}}\varphi\rangle + \widehat{\varphi}\nabla_{\bm{x}}\cdot\left(\lambda\bm{gg}^{\top}\nabla_{\bm{x}}\varphi\right)\nonumber\\
&= \varphi\dfrac{\partial\widehat{\varphi}}{\partial t} + \varphi\nabla_{\bm{x}}\cdot\left(\widehat{\varphi}\bm{f}\right) + \varphi\dfrac{q\widehat{\varphi}}{\lambda} - \dfrac{\widehat{\varphi}}{2}\langle\bm{\Sigma},\hess_{\bm{x}}\varphi\rangle \nonumber \\ &- \dfrac{\widehat{\varphi}}{2}\langle\nabla_{\bm{x}}\log\varphi,\left(\lambda\bm{gg}^{\top}-\bm{\Sigma}\right)\nabla_{\bm{x}}\varphi\rangle \nonumber \\ &+  \langle\nabla_{\bm{x}}\widehat{\varphi},\lambda\bm{gg}^{\top}\nabla_{\bm{x}}\varphi\rangle + \widehat{\varphi}\nabla_{\bm{x}}\cdot\left(\lambda\bm{gg}^{\top}\nabla_{\bm{x}}\varphi\right),
\label{LHSofphihatPDEInMaking}    
\end{align}
where the expression inside the curly braces is due to \eqref{varphiPDE}. Using \eqref{WeightedLaplacianOfProduct} in Lemma \ref{LemmaWeightedLaplacian}, the same substitutions unpack the RHS of \eqref{PrimalPDE} as: 
\begin{align}
{\mathrm{RHS}}&=\dfrac{1}{2}\Delta_{\bm{\Sigma}}\left(\varphi\widehat{\varphi}\right)\nonumber\\
&=\dfrac{1}{2}\varphi\widehat{\varphi}\langle\hess_{\bm{x}},\bm{\Sigma}\rangle + \dfrac{1}{2}\varphi\langle\bm{\Sigma},\hess_{\bm{x}}\widehat{\varphi}\rangle \nonumber \\ &+ \dfrac{1}{2}\widehat{\varphi}\langle\bm{\Sigma},\hess_{\bm{x}}\varphi\rangle + \varphi\langle\nabla_{\bm{x}}\cdot\bm{\Sigma},\nabla_{\bm{x}}\widehat{\varphi}\rangle \nonumber\\
&+ \widehat{\varphi}\langle\nabla_{\bm{x}}\cdot\bm{\Sigma},\nabla_{\bm{x}}\varphi\rangle+\langle\nabla_{\bm{x}}\varphi,\bm{\Sigma}\nabla_{\bm{x}}\widehat{\varphi}\rangle.
\label{RHSofphihatPDEInMaking} 
\end{align}
For convenience, we refer to the six summands appearing in \eqref{RHSofphihatPDEInMaking}, from left to right, as term 1 to term 6. Comparing \eqref{LHSofphihatPDEInMaking} with \eqref{RHSofphihatPDEInMaking}, we make few observations:
\begin{itemize}
\item term 3 in \eqref{RHSofphihatPDEInMaking} can be combined with the term $-\dfrac{\widehat{\varphi}}{2}\langle\bm{\Sigma},\hess_{\bm{x}}\varphi\rangle$ in \eqref{LHSofphihatPDEInMaking},

\item term 6 in \eqref{RHSofphihatPDEInMaking} can be combined with the term $\langle\nabla_{\bm{x}}\widehat{\varphi},\lambda\bm{gg}^{\top}\nabla_{\bm{x}}\varphi\rangle$ in \eqref{LHSofphihatPDEInMaking},

\item the sum of the terms 1, 2 and 4 in \eqref{RHSofphihatPDEInMaking} equals $\dfrac{1}{2}\varphi\Delta_{\bm{\Sigma}}\:\widehat{\varphi}$, thanks to \eqref{WeightedLaplacianSingleFunction} in Lemma \ref{LemmaWeightedLaplacian}. 
\end{itemize}
Consequently, equating \eqref{LHSofphihatPDEInMaking} with \eqref{RHSofphihatPDEInMaking}, we obtain
\begin{align}
&\varphi\bigg\{\dfrac{\partial\widehat{\varphi}}{\partial t} + \nabla_{\bm{x}}\cdot\left(\widehat{\varphi}\bm{f}\right) + \dfrac{q\widehat{\varphi}}{\lambda} - \dfrac{1}{2}\Delta_{\bm{\Sigma}}\:\widehat{\varphi}\bigg\} \nonumber \\ &= \dfrac{\widehat{\varphi}}{2}\langle\nabla_{\bm{x}}\log\varphi,\left(\lambda\bm{gg}^{\top}-\bm{\Sigma}\right)\nabla_{\bm{x}}\varphi\rangle - \widehat{\varphi}\langle\lambda\bm{gg}^{\top}-\bm{\Sigma},\hess_{\bm{x}}\varphi\rangle\nonumber\\
&-\langle\nabla_{\bm{x}}\widehat{\varphi},\left(\lambda\bm{gg}^{\top}-\bm{\Sigma}\right)\nabla_{\bm{x}}\varphi\rangle - \widehat{\varphi}\langle\nabla_{\bm{x}}\cdot\left(\lambda\bm{gg}^{\top}-\bm{\Sigma}\right),\nabla_{\bm{x}}\varphi\rangle.
\label{EquatingLHSandRHSofphihatPDE}    
\end{align}
Dviding both sides of \eqref{EquatingLHSandRHSofphihatPDE} by $\varphi$ (recall that $\varphi$ is positive everywhere by definition), we find
\begin{align}
\dfrac{\partial\widehat{\varphi}}{\partial t} &+ \nabla_{\bm{x}}\cdot\left(\widehat{\varphi}\bm{f}\right) + \dfrac{q\widehat{\varphi}}{\lambda} - \dfrac{1}{2}\Delta_{\bm{\Sigma}}\:\widehat{\varphi} \nonumber \\ \qquad \qquad= &\dfrac{\widehat{\varphi}}{2}\langle\nabla_{\bm{x}}\log\varphi,\left(\lambda\bm{gg}^{\top}-\bm{\Sigma}\right)\nabla_{\bm{x}}\log\varphi\rangle \nonumber \\  \qquad &- \widehat{\varphi}\langle\lambda\bm{gg}^{\top}-\bm{\Sigma},\frac{1}{\varphi}\hess_{\bm{x}}\varphi\rangle\nonumber\\
&-\langle\nabla_{\bm{x}}\widehat{\varphi},\left(\lambda\bm{gg}^{\top} -\bm{\Sigma}\right)\nabla_{\bm{x}}\log\varphi\rangle \nonumber \\ &- \widehat{\varphi}\langle\nabla_{\bm{x}}\cdot\left(\lambda\bm{gg}^{\top}-\bm{\Sigma}\right),\nabla_{\bm{x}}\log\varphi\rangle\nonumber\\
&= -\nabla_{\bm{x}}\cdot\left(\widehat{\varphi}\left(\lambda\bm{gg}^{\top}-\bm{\Sigma}\right)\nabla_{\bm{x}}\log\varphi\right) \nonumber \\ & \qquad- \dfrac{\widehat{\varphi}}{2}\langle\nabla_{\bm{x}}\log\varphi,\left(\lambda\bm{gg}^{\top}-\bm{\Sigma}\right)\nabla_{\bm{x}}\log\phi\rangle \nonumber\\
&\stackrel{\eqref{def_fphi},\eqref{def_qphi}}{=} -\nabla_{\bm{x}}\cdot\left(\widehat{\varphi}\bm{f}_{\varphi}\right) - q_{\varphi}\widehat{\varphi},
\label{phihatPDEsemifinal}    
\end{align}
which is exactly \eqref{varphihatPDE}. This completes the proof.
\end{proof}

\begin{remark}
We refer to $\left(\widehat{\varphi},\varphi\right)$ as Schr\"{o}dinger factors \cite{aebi1996schrodinger,caluya2021wasserstein} since they comprise a factorization of $\rho^{\bm{u}}_{\mathrm{opt}}$, see \eqref{OptimalPDFfromFactors}.
\end{remark}

We next provide interpretations for the derived \eqref{BoundaryCoupledForwardBackward}, \eqref{def_fphi} and \eqref{def_qphi}. 

\subsection{Interpretations }\label{subsec:Interpretations}

\subsubsection*{Interpretation of \eqref{BoundaryCoupledForwardBackward}} 
The specific form of \eqref{BoundaryCoupledForwardBackward} with $\bm{f}_{\varphi},q_{\varphi}$ as in \eqref{def_fphi} and \eqref{def_qphi}, respectively, are new w.r.t. the existing literature. The PDEs \eqref{varphihatPDE}-\eqref{varphiPDE} can be seen as the forward-backward Kolmogorov advection-diffusion-reaction PDEs with drift coefficient $\bm{f}+\bm{f}_{\varphi}$, diffusion coefficient $\bm{\sigma}$, and reaction rate $q + q_{\varphi}$. In the special case of input and noise channels being the same (i.e., $\bm{gg}^{\top} \propto \bm{\sigma\sigma}^{\top}$), they recover known results such as \cite[eq. (5.10)]{chen2021stochastic} or \cite[Thm. 2]{caluya2021wasserstein}. Then, the relation $\bm{gg}^{\top} \propto \bm{\sigma\sigma}^{\top}$ can be seen as a generalized Einstein relation \cite{hernandez1989equilibrium,chen2015fast},\cite[Sec. III-B]{halder2022stochastic}. 

When these channels do not coincide, the difference $$\lambda\bm{gg}^{\top}-\bm{\Sigma}$$quantifies the imbalance between the control/forcing authority and the noise fluctuation, and causes an excess additive drift \eqref{def_fphi} and an excess additive reaction rate \eqref{def_qphi}, to ensure the conservation of optimally controlled probability mass.

Although in a different context, Kappen \cite[Sec. 3.1]{kappen2005path} provided a physical interpretation of $\bm{gg}^{\top} \propto \bm{\sigma\sigma}^{\top}$: in the directions of low noise, optimal control \eqref{OptimalControlfromFactors} is expensive, and vice versa. A geometric way to think about the same is depicted in Fig. \ref{fig:ControlAndNoiseChannel}. It appeals to the basic fact that the matrix-matrix self-outer product $\bm{gg}^{\top}$ (resp. $\bm{\sigma\sigma}^{\top}$) is a projector onto the span of columns of $\bm{g}$ (resp. $\bm{\sigma}$).

\begin{figure}[t]
\centering
\includegraphics[width=0.5\linewidth]{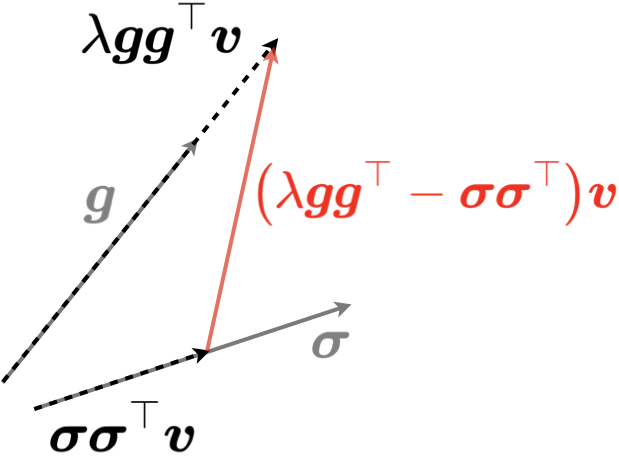}
\caption{When $m=p=1$, the vectors $\bm{g},\bm{\sigma}$ are shown as gray arrows. For fixed $\lambda>0$ and arbitrary nonzero vector $\bm{v}$, the dashed arrows are the projections $\lambda\bm{gg}^{\top}\bm{v}$ and $\bm{\sigma\sigma}^{\top}\bm{v}$ onto the span of $\bm{g}$ and $\bm{\sigma}$, respectively. When these projections become equal, their difference, shown in red, becomes zero.}
\label{fig:ControlAndNoiseChannel}
\end{figure}

\subsubsection*{Interpretation of \eqref{def_fphi}}
The excess additive drift $\bm{f}_{\varphi}$ in \eqref{def_fphi} is reminiscent of the excess additive drift in Doob's $h$-transform \cite{doob1957conditional}, \cite[Ch. IV-39]{rogers2000diffusions} that appears in conditioning an \emph{uncontrolled} It\^{o} diffusion \cite{pinsky1985convergence,deblassie1988doob,chetrite2015nonequilibrium}, and $\nabla_{\bm{x}}\log\varphi$ plays the role of score vector field \cite{song2019generative,song2020score,vahdat2021score}. The latter can in turn be thought of as a corrective advection 
to satisfy the endpoint constraints. However, the version \eqref{def_fphi} for control-affine Schr\"{o}dinger bridge is new. 

\subsubsection*{Interpretation of \eqref{def_qphi}}
The excess reaction rate $q_{\varphi}$ in \eqref{def_qphi} is reminiscent of the Fisher information that has appeared before in the Schr\"{o}dinger bridge problems as a state cost regularizer; see e.g., \cite{carlen1984conservative,leonard2013survey,carlen2014stochastic,chen2016relation,leger2019geometric}. However, the version appearing in \eqref{def_qphi} is different from these: it is weighted by (in general sign indefinite) $\lambda\bm{gg}^{\top}-\bm{\Sigma}$, and is specific to the control-affine case. Notably, the $q_{\varphi}$ here does not come from relating a Schr\"{o}dinger bridge problem with an equivalent optimal mass transport problem. Instead, it has an intrinsic control-theoretic origin: it results from the possible lack of mismatch between the input and noise channels.

\subsection{Implications }\label{subsec:Implications}
Notice that the PDEs \eqref{varphihatPDE}-\eqref{varphiPDE} are still \emph{coupled and nonlinear}. Similar to \eqref{PrimalPDE}-\eqref{DualPDE}, the equation-level coupling in \eqref{varphihatPDE}-\eqref{varphiPDE} is one way in the sense \eqref{varphiPDE} has no dependence on $\widehat{\varphi}$ but \eqref{varphihatPDE} depends on both $\varphi,\widehat{\varphi}$. Unlike \eqref{PrimalBC}, the boundary conditions \eqref{BilinearBC} are now bilinear. For these reasons, the utility of Hopf-Cole transform is not immediate. 

However, a closer inspection of \eqref{varphihatPDE}-\eqref{varphiPDE} reveals that the coupling and the nonlinearities emanate precisely from the $\varphi$-dependent vector field $\bm{f}_{\varphi}$ and the the $\varphi$-dependent reaction term $q_{\varphi}$ defined in \eqref{def_fphi} and \eqref{def_qphi}, respectively. The $\bm{f}_{\varphi}$ and $q_{\varphi}$, are both linear in the matrix $\lambda\bm{gg}^{\top}-\bm{\Sigma}$. So, these PDEs become \emph{decoupled and linear} provided the control and noise channels coincide. This is summarized in the following Corollary \ref{Corollary:DecoupledLinearPDEs}.

\begin{corollary}[Boundary-coupled linear PDEs]\label{Corollary:DecoupledLinearPDEs}
Consider the same setting as in Theorem \ref{Thm:ForwardBackward}. If  $\bm{gg}^{\top} \propto \bm{\sigma\sigma}^{\top}$ or equivalently, $\lambda\bm{gg}^{\top}- \bm{\Sigma}=\bm{0}$ for some $\lambda>0$, then \eqref{BoundaryCoupledForwardBackward} reduces to the following system of linear advection-diffusion-reaction PDEs that are only coupled through bilinear boundary conditions:
\begin{subequations}
\begin{align}
&\dfrac{\partial\widehat{\varphi}}{\partial t} + \nabla_{\bm{x}}\cdot\left(\widehat{\varphi}\bm{f}\right) - \dfrac{1}{2}\Delta_{\bm{\Sigma}}\:\widehat{\varphi} + \dfrac{q\widehat{\varphi}}{\lambda} = 0, \label{varphihatPDElinear}\\ 
&\dfrac{\partial\varphi}{\partial t} + \langle\nabla_{\bm{x}}\varphi,\bm{f}\rangle + \dfrac{1}{2}\langle\bm{\Sigma},\hess_{\bm{x}}\varphi\rangle -  \dfrac{q\varphi}{\lambda} = 0, \label{varphiPDElinear}\\
&\widehat{\varphi}\left(t_0,\cdot\right)\varphi\left(t_0,\cdot\right)=\rho_0(\cdot), \quad \widehat{\varphi}\left(t_1,\cdot\right)\varphi\left(t_1,\cdot\right)=\rho_1(\cdot). \label{BilinearBCAgain}
\end{align}
\label{BoundaryCoupledLinearForwardBackward}   
\end{subequations}
\end{corollary}
\begin{proof}
From \eqref{def_fphi}-\eqref{def_qphi}, both $\bm{f}_{\varphi},q_{\varphi}$ vanish for $\lambda\bm{gg}^{\top}- \bm{\Sigma}=\bm{0}$. Hence the result.
\end{proof}

The integral version of \eqref{BoundaryCoupledLinearForwardBackward} is a system of coupled nonlinear integral equations in $\left(\widehat{\varphi}(t_0,\cdot),\varphi(t_1,\cdot)\right)$, referred to as the \emph{Schr\"{o}dinger system}. Pioneering results showing the existence-uniqueness for the same are due to \cite{fortet1940resolution,beurling1960automorphism,jamison1974reciprocal}. 

\begin{remark}
Since the only coupling in \eqref{BoundaryCoupledLinearForwardBackward} comes from the boundary conditions, the system \eqref{BoundaryCoupledLinearForwardBackward} can be solved by performing fixed point recursions over the endpoint pair $\left(\widehat{\varphi}\left(t=t_0,\cdot\right),\varphi_{1}\left(t=t_1,\cdot\right)\right)$. Such recursions are now called dynamic Sinkhorn recursions and are known \cite{chen2016entropic,marino2020optimal} to be contractive w.r.t. Hilbert's projective metric \cite{bushell1973hilbert,lemmens2014birkhoff}, thereby guaranteeing worst-case linear convergence. Several generalizations of these results have appeared in the recent years \cite{caluya2021wasserstein,caluya2021reflected,chen2022most,teter2023contraction}.     
\end{remark}

This discussion raises an interesting possibility: whether the Sinkhorn-style computation can be adapted to solve the more general forward-backward nonlinear PDE system \eqref{BoundaryCoupledForwardBackward}. In principle, one may backward integrate \eqref{varphiPDE} with a terminal guess $\varphi(t_1,\cdot)$ from $t=t_1$ till $t=t_0$, saving the $\varphi(t,\cdot)$ for all $t\in[t_0,t_1]$. Then \eqref{varphihatPDE} can be forward integrated with the initial condition $\rho_0/\varphi(t=t_0,\cdot)$ and the data $\bm{f}_{\varphi},q_{\varphi}$ corresponding to the $\varphi(t,\cdot)$ from the backward pass, resulting in the updated terminal condition $\rho_1/\widehat{\varphi}(t=t_1,\cdot)$. Then these backward and forward passes can be repeated, as in the usual dynamic Sinkhorn. 

Apart form the new difficulty of integrating nonlinear \eqref{varphiPDE}, it is unclear if such generalized recursions may still be contractive in the projective sense. Algorithms to solve \eqref{BoundaryCoupledForwardBackward} will be explored in our future work.

%%%%%%%%%%%%%%%%%%%%%%%%%%%%%%%%%%%%%%%%%%
\section{Concluding Remarks}\label{sec:conclusions}
We considered generic control-affine Schr\"{o}dinger bridge problems with prior drift $\bm{f}$, control coefficient $\bm{g}$ and noise (i.e., diffusion) coefficient $\bm{\sigma}$. We showed that applying the Hopf-Cole transform to the necessary conditions of optimality for such problems lead to a boundary-coupled system of \emph{nonlinear} advection-reaction-diffusion PDEs. The nonlinearity is \emph{not} due to diffusion but due to an ``effective drift"
\begin{align}
\bm{\bm{f}}+\bm{f}_{\varphi} = \bm{\bm{f}}+\left(\lambda\bm{gg}^{\top}-\bm{\Sigma}\right)\nabla_{\bm{x}}\log\varphi,
\label{EffectiveDrift}    
\end{align}
and an ``effective reaction" rate
\begin{align}
q + q_{\varphi} = q +\dfrac{1}{2}\left(\nabla_{\bm{x}}\log\varphi\right)^{\top}\left(\lambda\bm{gg}^{\top}-\bm{\Sigma}\right)\nabla_{\bm{x}}\log\varphi,
\label{EffectiveReactionRate}    
\end{align}
where $q$ is the original state cost-to-go. For the same control and noise channels, the nonlinear terms $\bm{f}_{\varphi},q_{\varphi}$ become zero. In that special case, numerical solution via dynamic Sinkhorn recursions are well-known with contraction guarantees. Consistent with Schr\"{o}dinger's original motivation, the mathematical parallels with quantum mechanics remain intriguing: the relation \eqref{OptimalPDFfromFactors} is the stochastic analogue of Born's relation \cite{born1926quantenmechanik}. The excess drift $\bm{f}_{\varphi}$ in \eqref{def_fphi} appears as the stochastic analogue of the ``current velocity" in the de Broglie–Bohm interpretation \cite{bohm1952suggested}. Our developments here showed that the excess reaction rate $q_{\varphi}$ in \eqref{def_qphi} is a squared-magnitude of this current velocity w.r.t. a sign indefinite weight $\lambda\bm{gg}^{\top}-\bm{\Sigma}$, quantifying the input and noise channel mismatch. 

%%%%%%%%%%%%%%%%%%%%%%%%%%%%%%%%%%%%%%%%%%

\section*{Acknowledgment} 
During this work, the last author was supported in part through a Faculty Scholar program at the Physics Division in Lawrence Livermore National Laboratory (LLNL). The support and insightful discussions with Michael Schneider, Jane Pratt, Alexx Perloff, Maria Demireva and Conor Artman at LLNL are gratefully acknowledged.

%%%%%%%%%%%%%%%%%%%%%%%%%%%%%%%%%%%%%%%%%%
\balance
% \appendix

\section*{References}

\bibliographystyle{IEEEtran}
\bibliography{References.bib}

\begin{thebibliography}{10}
\providecommand{\url}[1]{#1}
\csname url@samestyle\endcsname
\providecommand{\newblock}{\relax}
\providecommand{\bibinfo}[2]{#2}
\providecommand{\BIBentrySTDinterwordspacing}{\spaceskip=0pt\relax}
\providecommand{\BIBentryALTinterwordstretchfactor}{4}
\providecommand{\BIBentryALTinterwordspacing}{\spaceskip=\fontdimen2\font plus
\BIBentryALTinterwordstretchfactor\fontdimen3\font minus \fontdimen4\font\relax}
\providecommand{\BIBforeignlanguage}[2]{{%
\expandafter\ifx\csname l@#1\endcsname\relax
\typeout{** WARNING: IEEEtran.bst: No hyphenation pattern has been}%
\typeout{** loaded for the language `#1'. Using the pattern for}%
\typeout{** the default language instead.}%
\else
\language=\csname l@#1\endcsname
\fi
#2}}
\providecommand{\BIBdecl}{\relax}
\BIBdecl

\bibitem{Sch31}
E.~Schr{\"o}dinger, ``{\"U}ber die {U}mkehrung der {N}aturgesetze,'' \emph{{S}itzungsberichte der {P}reuss {A}kad. {W}issen. {P}hys. {M}ath. {K}lasse, {S}onderausgabe}, vol.~IX, pp. 144--153, 1931.

\bibitem{Sch32}
E.~Schr{\"o}dinger, ``Sur la th{\'e}orie relativiste de l'{\'e}lectron et l'interpr{\'e}tation de la m{\'e}canique quantique,'' in \emph{Annales de L'Institut Henri Poincar{\'e}}, vol.~2, no.~4.\hskip 1em plus 0.5em minus 0.4em\relax Presses universitaires de France, 1932, pp. 269--310.

\bibitem{chetrite2021schrodinger}
R.~Chetrite, P.~Muratore-Ginanneschi, and K.~Schwieger, ``E. {S}chr{\"o}dinger’s 1931 paper {“On the Reversal of the Laws of Nature”[“{\"U}ber die Umkehrung der Naturgesetze”, Sitzungsberichte der preussischen Akademie der Wissenschaften, physikalisch-mathematische Klasse, 8 N9 144--153]},'' \emph{The European Physical Journal H}, vol.~46, pp. 1--29, 2021.

\bibitem{kolmogoroff1933grundbegriffe}
A.~Kolmogoroff, ``Grundbegriffe der wahrscheinlichkeitsrechnung,'' 1933.

\bibitem{aastrom2014control}
K.~J. {\AA}str{\"o}m and P.~R. Kumar, ``Control: A perspective.'' \emph{Automatica}, vol.~50, no.~1, pp. 3--43, 2014.

\bibitem{follmer1988random}
H.~F{\"o}llmer, ``Random fields and diffusion processes,'' \emph{Lect. Notes Math}, vol. 1362, pp. 101--204, 1988.

\bibitem{dai1991stochastic}
P.~Dai~Pra, ``A stochastic control approach to reciprocal diffusion processes,'' \emph{Applied Mathematics and Optimization}, vol.~23, no.~1, pp. 313--329, 1991.

\bibitem{pavon1991free}
M.~Pavon and A.~Wakolbinger, ``On free energy, stochastic control, and {S}chr{\"o}dinger processes,'' in \emph{Modeling, Estimation and Control of Systems with Uncertainty: Proceedings of a Conference held in Sopron, Hungary, September 1990}.\hskip 1em plus 0.5em minus 0.4em\relax Springer, 1991, pp. 334--348.

\bibitem{blaquiere1992controllability}
A.~Blaquiere, ``Controllability of a {F}okker-{P}lanck equation, the {S}chr{\"o}dinger system, and a related stochastic optimal control (revised version),'' \emph{Dynamics and Control}, vol.~2, no.~3, pp. 235--253, 1992.

\bibitem{chen2016relation}
Y.~Chen, T.~T. Georgiou, and M.~Pavon, ``On the relation between optimal transport and {S}chr{\"o}dinger bridges: A stochastic control viewpoint,'' \emph{Journal of Optimization Theory and Applications}, vol. 169, pp. 671--691, 2016.

\bibitem{chen2021stochastic}
Y.~Chen, T.~T. Georgiou, and M.~Pavon, ``Stochastic control liaisons: {Richard {S}inkhorn meets Gaspard Monge on a {S}chr{\"o}dinger} bridge,'' \emph{Siam Review}, vol.~63, no.~2, pp. 249--313, 2021.

\bibitem{dawson1990schrodinger}
D.~Dawson, L.~Gorostiza, and A.~Wakolbinger, ``Schr{\"o}dinger processes and large deviations,'' \emph{Journal of Mathematical Physics}, vol.~31, no.~10, pp. 2385--2388, 1990.

\bibitem{wakolbinger1990schrodinger}
A.~Wakolbinger \emph{et~al.}, ``Schr{\"o}dinger bridges from 1931 to 1991,'' in \emph{Proc. of the 4th Latin American Congress in Probability and Mathematical Statistics, Mexico City}, 1990, pp. 61--79.

\bibitem{leonard2011stochastic}
C.~L{\'e}onard, ``Stochastic derivatives and generalized h-transforms of {M}arkov processes,'' \emph{arXiv preprint arXiv:1102.3172}, 2011.

\bibitem{teter2024schr}
A.~M. Teter, W.~Wang, and A.~Halder, ``{S}chrödinger bridge with quadratic state cost is exactly solvable,'' \emph{arXiv preprint arXiv:2406.00503}, 2024.

\bibitem{teter2024weyl}
A.~M. Teter, W.~Wang, and A.~Halder, ``Weyl calculus and exactly solvable {S}chr{\"o}dinger bridges with quadratic state cost,'' in \emph{2024 60th Annual Allerton Conference on Communication, Control, and Computing}.\hskip 1em plus 0.5em minus 0.4em\relax IEEE, 2024, pp. 1--8.

\bibitem{teter2025probabilistic}
A.~M. Teter, I.~Nodozi, and A.~Halder, ``Probabilistic {L}ambert problem: Connections with optimal mass transport, {S}chr{\"o}dinger bridge, and reaction-diffusion {PDE}s,'' \emph{SIAM Journal on Applied Dynamical Systems}, vol.~24, no.~1, pp. 16--43, 2025.

\bibitem{chen2015fast}
Y.~Chen, T.~T. Georgiou, and M.~Pavon, ``Fast cooling for a system of stochastic oscillators,'' \emph{Journal of Mathematical Physics}, vol.~56, no.~11, 2015.

\bibitem{haddad2020prediction}
S.~Haddad, K.~F. Caluya, A.~Halder, and B.~Singh, ``Prediction and optimal feedback steering of probability density functions for safe automated driving,'' \emph{IEEE Control Systems Letters}, vol.~5, no.~6, pp. 2168--2173, 2020.

\bibitem{caluya2020finite}
K.~F. Caluya and A.~Halder, ``Finite horizon density steering for multi-input state feedback linearizable systems,'' in \emph{2020 American Control Conference (ACC)}.\hskip 1em plus 0.5em minus 0.4em\relax IEEE, 2020, pp. 3577--3582.

\bibitem{caluya2021wasserstein}
K.~F. Caluya and A.~Halder, ``Wasserstein proximal algorithms for the {S}chr{\"o}dinger bridge problem: Density control with nonlinear drift,'' \emph{IEEE Transactions on Automatic Control}, vol.~67, no.~3, pp. 1163--1178, 2021.

\bibitem{nodozi2022schrodinger}
I.~Nodozi and A.~Halder, ``Schr{\"o}dinger meets {K}uramoto via {F}eynman-{K}ac: Minimum effort distribution steering for noisy nonuniform {K}uramoto oscillators,'' in \emph{2022 IEEE 61st Conference on Decision and Control (CDC)}.\hskip 1em plus 0.5em minus 0.4em\relax IEEE, 2022, pp. 2953--2960.

\bibitem{behjoo2025harmonic}
H.~Behjoo and M.~Chertkov, ``Harmonic path integral diffusion,'' \emph{IEEE Access}, 2025.

\bibitem{caluya2021reflected}
K.~F. Caluya and A.~Halder, ``Reflected {S}chr{\"o}dinger bridge: Density control with path constraints,'' in \emph{2021 American Control Conference (ACC)}.\hskip 1em plus 0.5em minus 0.4em\relax IEEE, 2021, pp. 1137--1142.

\bibitem{hopf1950partial}
E.~Hopf, ``The partial differential equation $u_t+uu_x=\mu_{xx}$,'' \emph{Communications on Pure and Applied Mathematics}, vol.~3, no.~3, pp. 201--230, 1950.

\bibitem{cole1951quasi}
J.~D. Cole, ``On a quasi-linear parabolic equation occurring in aerodynamics,'' \emph{Quarterly of Applied Mathematics}, vol.~9, no.~3, pp. 225--236, 1951.

\bibitem{marino2020optimal}
S.~D. Marino and A.~Gerolin, ``An optimal transport approach for the {S}chr{\"o}dinger bridge problem and convergence of {S}inkhorn algorithm,'' \emph{Journal of Scientific Computing}, vol.~85, no.~2, p.~27, 2020.

\bibitem{chen2016entropic}
Y.~Chen, T.~Georgiou, and M.~Pavon, ``Entropic and displacement interpolation: a computational approach using the {H}ilbert metric,'' \emph{SIAM Journal on Applied Mathematics}, vol.~76, no.~6, pp. 2375--2396, 2016.

\bibitem{teter2023contraction}
A.~M. Teter, Y.~Chen, and A.~Halder, ``On the contraction coefficient of the {S}chr{\"o}dinger bridge for stochastic linear systems,'' \emph{IEEE Control Systems Letters}, vol.~7, pp. 3325--3330, 2023.

\bibitem{nodozi2023neural}
I.~Nodozi, C.~Yan, M.~Khare, A.~Halder, and A.~Mesbah, ``Neural {S}chr{\"o}dinger bridge with {S}inkhorn losses: Application to data-driven minimum effort control of colloidal self-assembly,'' \emph{IEEE Transactions on Control Systems Technology}, vol.~32, no.~3, pp. 960--973, 2023.

\bibitem{nodozi2023physics}
I.~Nodozi, J.~O’Leary, A.~Mesbah, and A.~Halder, ``A physics-informed deep learning approach for minimum effort stochastic control of colloidal self-assembly,'' in \emph{2023 American Control Conference (ACC)}.\hskip 1em plus 0.5em minus 0.4em\relax IEEE, 2023, pp. 609--615.

\bibitem{oksendal2013stochastic}
B.~Oksendal, \emph{Stochastic differential equations: an introduction with applications}.\hskip 1em plus 0.5em minus 0.4em\relax Springer Science \& Business Media, 2013.

\bibitem{karatzas2014brownian}
I.~Karatzas and S.~Shreve, \emph{Brownian motion and stochastic calculus}.\hskip 1em plus 0.5em minus 0.4em\relax springer, 2014, vol. 113.

\bibitem{fleming2005logarithmic}
W.~H. Fleming, ``Logarithmic transformations and stochastic control,'' in \emph{Advances in Filtering and Optimal Stochastic Control: Proceedings of the IFIP-WG 7/1 Working Conference Cocoyoc, Mexico, February 1--6, 1982}.\hskip 1em plus 0.5em minus 0.4em\relax Springer, 2005, pp. 131--141.

\bibitem{fleming1983stochastic}
W.~H. Fleming, ``Stochastic calculus of variations and mechanics,'' \emph{Journal of Optimization Theory and Applications}, vol.~41, no.~1, pp. 55--74, 1983.

\bibitem{kappen2005path}
H.~J. Kappen, ``Path integrals and symmetry breaking for optimal control theory,'' \emph{Journal of Statistical Mechanics: Theory and Experiment}, vol. 2005, no.~11, p. P11011, 2005.

\bibitem{todorov2009efficient}
E.~Todorov, ``Efficient computation of optimal actions,'' \emph{Proceedings of the National Academy of Sciences}, vol. 106, no.~28, pp. 11\,478--11\,483, 2009.

\bibitem{fleming1982optimal}
W.~H. Fleming and S.~K. Mitter, ``Optimal control and nonlinear filtering for nondegenerate diffusion processes,'' \emph{Stochastics: An International Journal of Probability and Stochastic Processes}, vol.~8, no.~1, pp. 63--77, 1982.

\bibitem{chetrite2015variational}
R.~Chetrite and H.~Touchette, ``Variational and optimal control representations of conditioned and driven processes,'' \emph{Journal of Statistical Mechanics: Theory and Experiment}, vol. 2015, no.~12, p. P12001, 2015.

\bibitem{bierkens2016linear}
J.~Bierkens, V.~Y. Chernyak, M.~Chertkov, and H.~J. Kappen, ``Linear {PDE}s and eigenvalue problems corresponding to ergodic stochastic optimization problems on compact manifolds,'' \emph{Journal of Statistical Mechanics: Theory and Experiment}, vol. 2016, no.~1, p. 013206, 2016.

\bibitem{aebi1996schrodinger}
R.~Aebi, \emph{{S}chr{\"o}dinger diffusion processes}.\hskip 1em plus 0.5em minus 0.4em\relax Springer Science \& Business Media, 1996.

\bibitem{hernandez1989equilibrium}
D.~B. Hernandez and M.~Pavon, ``Equilibrium description of a particle system in a heat bath,'' \emph{Acta Applicandae Mathematica}, vol.~14, pp. 239--258, 1989.

\bibitem{halder2022stochastic}
A.~Halder, K.~F. Caluya, P.~Ojaghi, and X.~Geng, ``Stochastic uncertainty propagation in power system dynamics using measure-valued proximal recursions,'' \emph{IEEE Transactions on Power Systems}, vol.~38, no.~5, pp. 4028--4041, 2022.

\bibitem{doob1957conditional}
J.~L. Doob, ``Conditional brownian motion and the boundary limits of harmonic functions,'' \emph{Bulletin de la Soci{\'e}t{\'e} math{\'e}matique de France}, vol.~85, pp. 431--458, 1957.

\bibitem{rogers2000diffusions}
L.~C.~G. Rogers and D.~Williams, \emph{Diffusions, Markov processes, and martingales: It{\^o} calculus}.\hskip 1em plus 0.5em minus 0.4em\relax Cambridge university press, 2000, vol.~2.

\bibitem{pinsky1985convergence}
R.~G. Pinsky, ``On the convergence of diffusion processes conditioned to remain in a bounded region for large time to limiting positive recurrent diffusion processes,'' \emph{The Annals of Probability}, pp. 363--378, 1985.

\bibitem{deblassie1988doob}
R.~D. DeBlassie, ``Doob's conditioned diffusions and their lifetimes,'' \emph{The Annals of Probability}, pp. 1063--1083, 1988.

\bibitem{chetrite2015nonequilibrium}
R.~Chetrite and H.~Touchette, ``Nonequilibrium {M}arkov processes conditioned on large deviations,'' in \emph{Annales Henri Poincar{\'e}}, vol.~16.\hskip 1em plus 0.5em minus 0.4em\relax Springer, 2015, pp. 2005--2057.

\bibitem{song2019generative}
Y.~Song and S.~Ermon, ``Generative modeling by estimating gradients of the data distribution,'' \emph{Advances in Neural Information Processing Systems}, vol.~32, 2019.

\bibitem{song2020score}
Y.~Song, J.~Sohl-Dickstein, D.~P. Kingma, A.~Kumar, S.~Ermon, and B.~Poole, ``Score-based generative modeling through stochastic differential equations,'' \emph{arXiv preprint arXiv:2011.13456}, 2020.

\bibitem{vahdat2021score}
A.~Vahdat, K.~Kreis, and J.~Kautz, ``Score-based generative modeling in latent space,'' \emph{Advances in Neural Information Processing Systems}, vol.~34, pp. 11\,287--11\,302, 2021.

\bibitem{carlen1984conservative}
E.~A. Carlen, ``Conservative diffusions,'' \emph{Communications in Mathematical Physics}, vol.~94, pp. 293--315, 1984.

\bibitem{leonard2013survey}
C.~L{\'e}onard, ``A survey of the {S}chr{\"o}dinger problem and some of its connections with optimal transport,'' \emph{Discrete and Continuous Dynamical Systems}, vol.~34, no.~4, pp. 1533--1574, 2013.

\bibitem{carlen2014stochastic}
E.~Carlen, ``Stochastic mechanics: A look back and a look ahead,'' \emph{Diffusion, Quantum Theory and Radically Elementary Mathematics}, vol.~47, pp. 117--139, 2014.

\bibitem{leger2019geometric}
F.~L{\'e}ger, ``A geometric perspective on regularized optimal transport,'' \emph{Journal of Dynamics and Differential Equations}, vol.~31, no.~4, pp. 1777--1791, 2019.

\bibitem{fortet1940resolution}
R.~Fortet, ``R{\'e}solution d'un syst{\`e}me d'{\'e}quations de m. {S}chr{\"o}dinger,'' \emph{Journal de Math{\'e}matiques Pures et Appliqu{\'e}es}, vol.~19, no. 1-4, pp. 83--105, 1940.

\bibitem{beurling1960automorphism}
A.~Beurling, ``An automorphism of product measures,'' \emph{Annals of Mathematics}, vol.~72, no.~1, pp. 189--200, 1960.

\bibitem{jamison1974reciprocal}
B.~Jamison, ``Reciprocal processes,'' \emph{Zeitschrift f{\"u}r Wahrscheinlichkeitstheorie und Verwandte Gebiete}, vol.~30, no.~1, pp. 65--86, 1974.

\bibitem{bushell1973hilbert}
P.~J. Bushell, ``Hilbert's metric and positive contraction mappings in a {B}anach space,'' \emph{Archive for Rational Mechanics and Analysis}, vol.~52, pp. 330--338, 1973.

\bibitem{lemmens2014birkhoff}
B.~Lemmens and R.~Nussbaum, ``Birkhoff’s version of {H}ilbert’s metric and its applications in analysis,'' \emph{Handbook of Hilbert Geometry}, pp. 275--303, 2014.

\bibitem{chen2022most}
Y.~Chen, T.~T. Georgiou, and M.~Pavon, ``The most likely evolution of diffusing and vanishing particles: {S}chrödinger bridges with unbalanced marginals,'' \emph{SIAM Journal on Control and Optimization}, vol.~60, no.~4, pp. 2016--2039, 2022.

\bibitem{born1926quantenmechanik}
M.~Born, ``Quantenmechanik der sto{\ss}vorg{\"a}nge,'' \emph{Zeitschrift f{\"u}r physik}, vol.~38, no.~11, pp. 803--827, 1926.

\bibitem{bohm1952suggested}
D.~Bohm, ``A suggested interpretation of the quantum theory in terms of "hidden" variables. {I},'' \emph{Physical review}, vol.~85, no.~2, p. 166, 1952.

\end{thebibliography}

\end{document}